\documentclass{amsart}

\usepackage{amsmath, amsfonts,amssymb,amsthm}
\usepackage{graphicx}
\usepackage{thmtools,thm-restate}
\usepackage{enumitem}
\newtheorem{thm}{Theorem}
\newtheorem{lemma}[thm]{Lemma}
\newtheorem{cor}[thm]{Corollary}
\newtheorem{conjecture}{Conjecture}

\theoremstyle{definition}

\newenvironment{claim}[2]{\vspace{0.2cm}\par\noindent\emph{Claim $#1$.}\space#2}{}

\newcommand{\N}{\mathbb{N}}
\newcommand{\K}[1]{\vec{K}_{#1} }
\newcommand{\A}{\mathcal{A}}
\newcommand{\U}{\mathcal{U}}

\newcommand{\singleton}[1]{\{#1\}}
\newcommand{\set}[1]{\{#1\}}

\begin{document}
\author{Carl B\"urger}
\author{Max Pitz}
\address{University of Hamburg, Department of Mathematics, Bundesstra{\ss}e 55 (Geomatikum), 20146 Hamburg, Germany}
\email{carl.buerger@uni-hamburg.de, max.pitz@uni-hamburg.de}

\title[Decomposing complete symmetric digraphs into monochromatic paths]{Decomposing edge-coloured complete symmetric digraphs into monochromatic paths}  



\begin{abstract}
Confirming and extending a conjecture by Guggiari, we show that every countable $(r+1)$-edge-coloured complete symmetric digraph containing no directed paths of edge-length $\ell_i$ for any colour $i\leq r$ can be covered by $\prod_{i\leq r} \ell_i$ pairwise disjoint monochromatic directed paths in colour $r+1$.   
\end{abstract}

\maketitle
\section{Introduction}
For $n \in \N$ we write $[n] = \{1,2,\ldots, n \}$. Let $K_n$ be the complete undirected graph with vertex set $[n]$, and $\K{n}$ be the complete symmetric digraph on $[n]$, i.e.\ the directed graph where every edge of $K_n$ appears in both its orientations. A \emph{tournament} of order $n$ is a complete antisymmetric digraph on $[n]$, i.e.\ a directed graph in which every edge of $K_n$ appears in precisely one of its possible orientations. Similarly, let $K_\N$ and $\K{\N}$ be the complete graph and the complete symmetric digraph on the positive integers respectively.

Let $G$ be a digraph. A sequence $v_1, v_2, \ldots,v_{n}$ of vertices such that there is an oriented edge $\vec{e}_i = (v_i,v_{i+1}) \in E(G)$ from $v_i$ to $v_{i+1}$ for all $i \in [n-1]$ is a \emph{directed} path in $G$ of \emph{length} $n-1$ (in this paper, the length of a path always refers to its edge-length). A sequence $v_1,v_2,v_3, \ldots$ satisfying the above conditions for all $i \in \N$ is a \emph{one-way infinite directed path}. The term \emph{directed path} may refer to both finite and one-way infinite directed paths.

The \emph{upper density} of a set $A \subseteq \N$ is defined as
\[\overline{d}(A) = \limsup_{n \to \infty} \frac{|A \cap [n]|}{n}.\]

DeBiasio and McKenney \cite{debiasio2016density} have recently shown that for every $\varepsilon > 0$, there exists a $2$-edge-colouring of $\K{\N}$ such that every monochromatic directed path has upper density less than $\varepsilon$. Answering a question of the above-named authors \cite[Problem 8.3]{debiasio2016density}, Guggiari \cite{Guggiari} constructed a $2$-edge-colouring of $\K{\N}$ such that every monochromatic path has upper density $0$, but observed that if one restricts the maximal length of directed paths in the first colour, then there must be monochromatic paths in the second colour with non-vanishing upper density. More generally:

\begin{thm}[Guggiari]
\label{thm_guggiariupperdensity}
For any edge colouring $c\colon E(\K{\N}) \to [r+1]$ for which there are no directed paths of length $\ell_i$ in colour $i$ for any $i\in [r]$, there is a monochromatic directed path in colour $r+1$ with upper density at least $\prod_{i\leq r} \ell_i^{-1}$.
\end{thm}

After establishing the upper-density result, Guggiari concludes her paper with the following conjecture:

\begin{conjecture}[Guggiari]
Take any $2$-edge-colouring of $\K{\N}$ that does not contain a red directed path of length $\ell$. Then the vertices of $\K{\N}$ can be covered by at most $\ell$ vertex-disjoint blue directed paths.
\end{conjecture}

Clearly, such a partition result would imply the corresponding upper density result, for if one has a partition into at most $\ell$ vertex-disjoint directed paths, then by the pigeon hole principle, one of the paths must have upper density at least $1/\ell$.

The purpose of this note is to confirm Guggiari's conjecture, and furthermore to extend it to all finite edge-colourings of $\K{\N}$, with possibly more than two colours.

\begin{thm}\label{main decomposition thm} 
Let $G$ be a complete symmetric digraph, either finite or countably infinite. Then for every $(r+1)$-edge-colouring for which there is no monochromatic directed path of edge-length $\ell_i$ in colour $i$ for any $i\in [r]$, the vertex set of $G$ can be covered by $\prod_{i\leq r} \ell_i$ pairwise disjoint monochromatic directed paths in colour $r+1$.  
\end{thm}

Our proof contains three new ideas. First, following a proof strategy of Loh \cite{Po-ShenLoh}, we show in Section~\ref{sec3} that edge-coloured tournaments without long monochromatic paths are somewhat small. Second, in Section~\ref{sec4}, we show that for a finitely edge-coloured $\K{n}$, the number $k$ of monochromatic paths in the last colour needed to cover the vertex set gives rise to a subtournament of $\K{n}$ of order $k$ where the last colour does not occur. 

These two observations will then be used as follows: In Section~\ref{sec5}, following the proof strategy by Guggiari \cite{Guggiari}, we construct a well-behaved finite partition $\U$ of $\N$ such that each partition classes can be covered by a monochromatic path in the last colour. The third and final new piece is in Section~\ref{sec6} to construct a subtournament of $\K{\N}$ not using the last colour which (almost) spans $\U$ --- allowing us to apply observations one and two to bound the number of partition classes in $\U$, giving rise to a decomposition into few monochromatic paths in the last colour. 

We remark that for the infinite case $\K{\N}$, the upper bound of $\prod_{i\leq r} \ell_i$ in Theorem~\ref{main decomposition thm} is best possible by the example in \cite[Figure~5]{Guggiari}. It would be interesting to know whether in the finite case the bound can be improved.

\section{Monochromatic paths in finite edge-coloured tournaments}
\label{sec3}

Before investigating complete symmetric digraphs, we begin our discussion with a  result on longest monochromatic paths in finitely edge-coloured tournaments.

One of the main results in a recent paper by Loh \cite{Po-ShenLoh} is that every $r$-edge coloured tournament of order $n$ contains a monochromatic directed path of length at least $\sqrt[\leftroot{2}\uproot{2}r]{n}-1$.\footnote{Loh defines the length of a path to be its \emph{vertex-length}, and so we have to add an additional `$-1$' in our statements, translating between  edge- and vertex-length.} The same proof method as in \cite[Theorem 1.1]{Po-ShenLoh}, with only minor modifications, also implies the following imbalanced version of this result:

\begin{thm}\label{po-shenloh thm}
Let $T$  be a finite tournament and $c\colon E(T)\rightarrow [r]$ an edge-colouring. Moreover, assume that there is no directed path of length $\ell_i$ in colour $i$ for $i\in [r]$. Then $T$ has order at most $\prod_{i\leq r} \ell_i$.
\end{thm}

The proof relies on the following theorem, which, as noted in \cite{Po-ShenLoh}, has been independently discovered in \cite{gallai1968directed, hasse1965algebraischen, roy1967nombre, vitaver1962determination}.

\begin{thm}[Gallai-Hasse-Roy-Vitaver]
\label{thm_GHRV}
Let $G$ be an undirected graph with chromatic number $k$. No matter how its edges are oriented, the resulting directed graph contains a directed path of length at least $k-1$. \qed
\end{thm}

\begin{proof}[Proof of Theorem~\ref{po-shenloh thm}]
Let $T$ be a tournament of order $n$, and consider an arbitrary edge-colouring $c\colon E(T)\rightarrow [r]$ of $T$ such that there is no directed path of length $\ell_i$ in colour $i$ for any $i\in [r]$.

This colouring induces an edge-coloring of the underlying undirected complete graph $K_n$, which partitions $K_n$ into $r$ edge-disjoint
subgraphs $G_1, G_2, \ldots, G_{r}$, each corresponding to a color class in the tournament. If $\chi(G)$ denotes the chromatic number of a graph $G$, then Theorem~\ref{thm_GHRV} and our assumptions on the various monochromatic path lengths imply that $\chi(G_i) \leq \ell_i$ for all $i$. By considering the product colouring on $K_n$, it follows that
\[n = \chi(K_n) \leq \prod_{i\leq r} \chi(G_i) \leq \prod_{i\leq r} \ell_i. \qedhere \]
\end{proof}

\section{Decomposing finite symmetric complete digraphs}
\label{sec4}

In this section, we use Theorem~\ref{po-shenloh thm} to prove our main decomposition theorem in the finite case. As a corollary to its proof, we make the useful structural observation that the number $k$ of monochromatic paths in the last colour needed to cover the vertex set gives rise to a subtournament of $\K{n}$ of order $k$ where the last colour does not occur---an observation which will be crucial further below in the proof our main result for the infinite case.

\begin{thm}\label{finite decomposition thm}
For any edge-colouring $c\colon E(\K{n})\rightarrow [r+1]$ such that there is no directed path of edge-length $\ell_i$ in colour $i$ for $i \in [r]$, the vertex set of $\K{n}$ can be covered by $\prod_{i
\leq r} \ell_i$ pairwise disjoint monochromatic directed paths in colour $r+1$.  
\end{thm}

\begin{proof}
Let $k$ be the smallest integer such that the vertex set of $\K{n}$ can be covered by monochromatic pairwise disjoint directed paths in colour $(r+1)$. Choose a sequence $\mathcal{P}:=(P_1,\dots,P_{k})$ of such monochromatic paths satisfying that $(|P_1|,\ldots,|P_{k}|)$ is minimal with respect to the lexicographical order of $\N^k$. Let us write $v_i$ for the first vertex on $P_i$ for $i\leq k$. Then $(v_i,v_j)$ has some colour $\leq r$, whenever $i<j$. Otherwise, the path system $$(P_1,\ldots,\mathring{v}_iP_i,\ldots, v_iP_j,\ldots,P_{k})$$ would contradict the minimality
of $\mathcal{P}$. Hence, the set $\{(v_i,v_j) \colon i<j\}$ defines a tournament $T$ on $\{v_1,\dots, v_k\}$ with edge colours in $[r]$. By Theorem~\ref{po-shenloh thm}, we have $k\le \prod_{i\leq r} \ell_i$, completing the proof. 
\end{proof}

As a consequence to the proof of Theorem \ref{finite decomposition thm} we obtain the following corollary:

\begin{cor}\label{cor finite decomposition thm}
Let $c\colon E(\K{n})\rightarrow [r+1]$ be an edge-colouring of $\K{n}$. If $k$ is the smallest number of pairwise disjoint monochromatic directed paths in colour $r+1$ needed to cover the vertex set of $\K{n}$, then $\K{n}$ contains a tournament of order $k$ with edge-colours in $[r]$ as a subgraph. \qed
\end{cor}

\section{A partition result}
\label{sec5}

A non-trivial digraph $G$ is said to be \emph{strongly ${<}\aleph_0$-connected} if for every two vertices $v$ and $w$ of $G$, there are infinitely many independent directed $v-w$ paths, and also infinitely many independent directed $w-v$ paths (equivalently: after deleting any finite number of vertices from $G-\{v,w\}$, the vertices $v$ and $w$ lie in the same strongly connected component).

In this section we gather two technical results. The first can be extracted from the proof of \cite[Theorem 1.3]{Guggiari}.

\begin{lemma}\label{lemma strongly connected then path}
Every countable strongly ${<}\aleph_0$-connected digraph can be covered by a single one-way infinite directed path. \qed
\end{lemma}

The second of our technical results is inspired by the techniques used in the proofs of \cite[Theorem 1.3 \& 1.5]{Guggiari}, but requires some modifications, and therefore we give the complete proof.

\begin{lemma}\label{lemma for infinite decomposition thm}
For every edge-colouring $c\colon E(\K{\N})\rightarrow [r+1]$ of $\K{\N}$ such that there is no directed path of edge-length $\ell_i$ in colour $i$ for $i \in [r]$, there is a finite partition $\U$ of $\N$ such that the edges coloured with colour $r+1$ induce a strongly ${<}\aleph_0$-connected subgraph on every non-singleton partition class of $\U$.
\end{lemma}

\begin{proof}
For $k=1,2,\ldots, r+1$ we will recursively define finite partitions $\U_k$ of $\N$ such that
\begin{enumerate}
\item $\U_{k}$ refines $\U_{k-1}$, 
\item\label{item3} for every $U\in \U_k$ and every vertex $u\in U$, the set $\{u'\in U\colon c(u',u)<k \}$ has cardinality at most $\sum_{j<k} \ell_j$, 
\item\label{item4} for every non-singleton partition class $U \in \U_{k}$, all vertices $u\in U$ satisfy that $\{u'\in U\colon c(u,u') \geq k \}$ is infinite.
\end{enumerate} 

Before we describe the recursive construction, let us see that $\U:=\U_{r+1}$ is as desired. Let $U \in \U$ be a non-singleton partition class, and let $v \neq w \in U$. To see that $U$ is ${<}\aleph_0$-connected, let us write $N^+_{r+1}(v)$ and $N^-_{r+1}(v)$ for the out- and in-neighbourhood of $v$ in colour $r+1$ respectively. By property~(\ref{item4}), we have $N^+_{r+1}(v) \cap U$ is infinite, and by property~(\ref{item3}), $N^-_{r+1}(w) \cap U$ is cofinite in $U$. Therefore, 
\[N^+_{r+1}(v) \cap N^-_{r+1}(w) \cap U \]
is infinite, and so there are infinitely many independent monochromatic directed $v-w$-paths in $U$ (each of length $2$) in colour $r+1$.

We now proceed with the recursive construction of the $\U_k$. For $k=1$, properties $(1)-(3)$ are trivially satisfied for $\U_1:= \{\N\}$. Now, assume that $\U_k$ has already been defined. For $0\leq i < \ell_k$ let $A_i$ consist of those vertices $a\in \N$ such that the longest $k$-coloured directed path in $\K{\N}$ with first vertex $a$ has length $i$. Then define $\A:=\{A_i\colon 0 \leq i < \ell_k\}\backslash \{\emptyset\}$. 

\begin{claim}{1}
$\A$ is a finite partition of $\N$ and for every partition class $A\in\A$ and every vertex $a\in A$ we have
$\{a'\in A\colon c(a',a)=k\}$ has cardinality at most $\ell_k$.
\end{claim}

\begin{proof}[Proof of Claim~1]\renewcommand{\qedsymbol}{$\Diamond$}
Cf.\ \cite[Theorem 1.3, Claim 1]{Guggiari}. Since $\K{\N}$ contains no directed monochromatic path in colour $k$ of length $\ell_k$ or bigger, it is clear that $\mathcal{A}$ is indeed a finite partition. Suppose for some $a \in A \in \A$ we have $|N^-_k(a) \cap A| > \ell_k$. Consider a longest monochromatic path $P$ in $\K{\N}$ in colour $k$ with first vertex $a$. Then $(N^-_k(a) \cap A) \setminus P \neq \emptyset$, and so for any $ a' \in (N^-_k(a) \cap A) \setminus P$, the path $P'=a'P$ is a strictly longer monochromatic path in colour $k$ starting at an element in the same partition class $A \in \A$, contradicting the definition of $\A$.
\end{proof}

Now consider the smallest common refinement $\mathcal{V} := \{U\cap A \colon U\in \U_k, A\in \A\}$.

\begin{claim}{2}
For every $V\in \mathcal{V}$ and every vertex $v\in V$, the set $\{v'\in V\colon c(v',v)<k+1 \}$ has cardinality at most $\sum_{j<k+1} \ell_j$.
\end{claim}

\begin{proof}[Proof of Claim~2]\renewcommand{\qedsymbol}{$\Diamond$}
Fix $v \in V\in \mathcal{V}$. Then $V = U \cap A$ for some $U \in \U_k$ and $A \in \A$. By property~(\ref{item3}), we have $\{u'\in U\colon c(u',v)<k \}$ has cardinality at most $\sum_{j<k} \ell_j$, and by Claim~1 we have $\{a'\in A\colon c(a',v)=k\}$ has cardinality at most $\ell_k$. Hence, $\{v'\in U \cap A \colon c(v',v)<k+1 \}$ has cardinality at most $\sum_{j<k+1} \ell_j$. 
\end{proof}

For $V\in\mathcal{V}$, let us define $X(V) = \{v \in V \colon |\{ v' \in V \colon c(v,v') \geq k+1\} | \text{ is finite}\}$.

\begin{claim}{3}
For every infinite partition class of $V\in\mathcal{V}$ the set $X(V)$ has cardinality at most $\sum_{j<k+1} \ell_j$.
\end{claim}

\begin{proof}[Proof of Claim~3]\renewcommand{\qedsymbol}{$\Diamond$}
Cf.\ \cite[Theorem 1.3, Claim 2]{Guggiari}. Indeed, consider some infinite $V \in \mathcal{V}$ and suppose for a contradiction that there is a finite subset $X\subseteq X(V)$ with $|X|>\sum_{j<k+1} \ell_j$. Because $V$ is infinite, there is a vertex:
\[w \in V \setminus \bigcup_{x \in X} \{ v' \in V \colon c(x,v') \geq k+1\}.\]
Then $c(x,w)<k+1$ for all $x \in X$, and so it follows from Claim~2 applied to the vertex $w \in V$ that
\[\sum_{j<k+1} \ell_j<|X| \leq |\{v'\in V\colon c(v',w)<k+1 \}| \leq \sum_{j<k+1} \ell_j, \]
a contradiction. 
\end{proof}

Finally, let $\mathcal{S}$ be the collection of singletons of the form $\{v\}$ for which $v$ is either part of a finite partition class of $\mathcal{V}$, or is contained in a set $X(V)$ for some $V \in \mathcal{V}$. By induction assumption and Claim~3, we know that $\mathcal{S}$ is finite. 

\begin{claim}{4}
$\U_{k+1}:=\mathcal{S}\cup \{V\backslash X(V)\colon V\in\mathcal{V} \text{ is infinite}\}$ satisfies properties $(1)-(3)$.
\end{claim}

\begin{proof}[Proof of Claim~4]\renewcommand{\qedsymbol}{$\Diamond$}
By construction, $\U_{k+1}$ is a finite partition of $\N$ such that every non-singleton partition class is infinite. Property~(1) for $\U_{k+1}$ is obvious. 
Property~(\ref{item3}) follows from Claim~2, as $\U_{k+1}$ is a refinement of $\mathcal{V}$. 
Finally, for property~(\ref{item4}) consider some infinite $U \in \U_{k+1}$. Then $U = V \setminus X(V)$ for some infinite $V \in \mathcal{V}$. By definition of $X(V)$, it follows that for all $u \in U$ the
\[ \{v \in V \colon c(u,v) \geq k+1 \} \]
is infinite, and therefore, as $X(V)$ is finite by Claim~3, we also have that 
\[ \{v \in V \setminus X(V) \colon c(u,v) \geq k+1 \} \]
is infinite, which verifies property~(\ref{item4}) for the partition class $U \in \U_{k+1}$.
\end{proof}
Thus, we see that $\U_{k+1}$ is as required, completing the recursive construction.
\end{proof}

\section{Decomposing countably infinite symmetric complete digraphs}
\label{sec6}

\begin{thm}\label{infinite decomposition thm} 
For every edge-colouring $c\colon E(\K{\N})\rightarrow [r+1]$ of $\K{\N}$ for which there is no directed path of length $\ell_i$ in colour $i$ for any $i \in [r]$, the vertex set $\N$ can be covered by $\prod_{i\leq r} \ell_i$ pairwise disjoint monochromatic directed paths in colour $r+1$.
\end{thm}

\begin{proof}
By Lemma~\ref{lemma for infinite decomposition thm} there exists a finite partition $\U$ of the vertex set $\N$ such that every partition class is either a singleton or strongly ${<}\aleph_0$-connected in colour $r+1$. Choose such a $\U$ with minimal cardinality. 

\begin{claim}{1}
Let $U,U'\in\U$ be infinite partition classes. Moreover, suppose that $M$ is a maximal directed $U-U'$ matching in colour $r+1$ and $M'$ is a maximal directed $U'-U$ matching in colour $r+1$. Then $M$ or $M'$ is finite.  
\end{claim}

\begin{proof}[Proof of Claim~1]\renewcommand{\qedsymbol}{$\Diamond$}
Suppose for a contradiction that $U,U'\in \U$ are both infinite partition classes such that there exist infinite directed matchings $M$ and $M'$ as in Claim 1. We aim to show that $U\cup U'$ is strongly ${<}\aleph_0$-connected in colour $r+1$, contradicting the minimal choice of $\U$. 

Towards this, fix vertices $v \in U$ and $w \in U'$. 
Since $U$ is ${<}\aleph_0$-connected in colour $r+1$ we find a directed $v-A$ fan in colour $r+1$ for some infinite $A\subseteq \{u\in U\colon (u,u')\in M\}$. Let $B=\{u'\in U'\colon u \in A, (u,u')\in M\}$. Similarly, since $U'$ is ${<}\aleph_0$-connected in colour $r+1$, there also exists a directed $B'-w$ fan for some infinite $B' \subset B$. Combining these two fans with suitable edges from $M$ shows that there are infinitely many independent monochromatic directed $v-w$-paths in colour $r+1$ in $U\cup U'$. 

Applying the same argument to the matching $M'$, one also finds infinitely many independent monochromatic directed $w-v$-paths in colour $r+1$ in $U\cup U'$. 
\end{proof}

Recall that we write $N^+_{r+1}(v)$ and $N^-_{r+1}(v)$ for the out- and in-neighbourhood respectively of a vertex $v\in\N$ in colour $r+1$. 
\begin{claim}{2}
Let $U,U'\in\U$ be partition classes such that $U=\{u\}$ is a singleton and $U'$ is infinite. Then $N^+_{r+1}(u)\cap U'$ or $N^-_{r+1}(u)\cap U'$ is finite. 
\end{claim}

\begin{proof}[Proof of Claim~2]\renewcommand{\qedsymbol}{$\Diamond$}
Suppose for a contradiction that there are partition classes $U,U'\in \U$ with $U=\singleton{u}$ and $U'$ infinite such that there are infinite sets $\vec{E}$ and $\vec{E}'$ of directed $u-U'$ and $U'-u$ edges respectively in colour $r+1$. Again, we claim that $U\cup U'$ is strongly ${<}\aleph_0$-connected in colour $r+1$, contradicting the minimal choice of $\U$. 


Towards this, fix a vertex $w \in U'$. Let $A= \{u' \in U'\colon (u,u')\in \vec{E}\}$. Since $U'$ is strongly ${<}\aleph_0$-connected in colour $r+1$, there is a directed $A'-w$ fan in colour $r+1$ for some infinite $A' \subseteq A$. Combining the fan with suitable edges from $\vec{E}$ shows that there are infinitely many independent monochromatic directed $u-w$-paths in colour $r+1$ in $U\cup U'$.   

Applying the same argument to set $\vec{E}'$ of edges, one also finds infinitely many independent monochromatic directed $w-u$-paths in colour $r+1$ in $U\cup U'$. 
\end{proof}

Let $\mathcal{S}$ be the set of singletons in $\U$ and $S:=\bigcup \mathcal{S}$. Furthermore, let $k$ be the smallest number of pairwise disjoint monochromatic directed paths in colour $r+~1$ needed, to cover $S$ in $\K{\N}[S]$ with regard to the colouring induced by $c$. By Corollary~\ref{cor finite decomposition thm}, there exists a sub-tournament $T'$ of order $k$ in $\K{\N}[S]$ with edge-colours in $[r]$. Let $\mathcal{S}':=\set{\singleton{v}  \colon   v\in V(T')}$. 

We now define a digraph $H$ with vertex set $\mathcal{S}'\cup (\U\backslash\mathcal{S})$, where we insert a direct edge $\vec{e}=(U,U')\in E(H)$ and define a finite set $W_{\vec{e}} \subset \N$ for every such edge, if
\begin{enumerate}[label={\upshape($\dagger$\arabic{*})}]
\item\label{bla1} both $U$ and $U'$ are infinite and there exists a maximal finite directed $U-U'$ matching $M_{\vec{e}}$ in colour $r+1$. In this case, define $W_{\vec{e}}:=V[M_{\vec{e}}]$ to be the set of vertices covered by $M_{\vec{e}}$.
\item\label{bla2} If $U=\{u\}$ is a singleton, $U'$ is infinite and $W_{\vec{e}}:=N^+_{r+1}(u) \cap U'$ is finite.
\item\label{bla3} If $U$ is infinite, $U'=\{u'\}$ is a singleton and $W_{\vec{e}}:=N^-_{r+1}(u') \cap U$ is finite.   
\item\label{bla4} Or if $U=\{u\}$ and $U'=\{u'\}$ are both singletons and $(u,u')\in E(T')$. In this case, we put $W_{\vec{e}} := \emptyset$. 
\end{enumerate}

Then we define
\[W = \bigcup_{\vec{e} \in E(H)} W_{\vec{e}}\]
which by construction is a finite subset of the vertex set $\N$ of $\K{\N}$.

Choose a vertex $x_U \in U \setminus W$ for every partition class $U\in \mathcal{S}'\cup(\U\backslash\mathcal{S})$ and write ${X=\set{x_U\colon U\in \mathcal{S}'\cup(\U\backslash\mathcal{S})}}$. Since $W$ is finite, and by construction has non-trivial intersection only with the infinite partition classes in $\U\backslash\mathcal{S}$, this is always possible. 

\begin{claim}{3}
The digraph $\K{\N}[X]$ induced by the vertices of $X$ contains a spanning sub-tournament $T$ with edge colours in $[r]$. 
\end{claim}

\begin{proof}[Proof of Claim~3]\renewcommand{\qedsymbol}{$\Diamond$}
We show that for every pair of vertices $x,x'\in X$, one of the directed edges $(x,x')$ or $(x',x)$ has a colour in $[r]$. Suppose $x \in U\in \U$ and $x' \in U'\in \U$. Since $T'$ is a tournament with edge colours in $[r]$, we may assume that one of $U$ and $U'$ infinite, and so it follows from Claim 1 and 2 that $(U,U')$ or $(U',U)$ is a directed edge of $H$. By symmetry, let us assume that $\vec{e}:=(U,U')\in E(H)$. 

\begin{itemize}
\item If $U$ and $U'$ are both infinite, then $(x,x')$ has a colour in $[r]$ because otherwise, as $x,x' \notin W$, the directed matching $M_{\vec{e}}\cup \set{(x,x')}$ contradicts the maximality of $M_{\vec{e}}$ in \ref{bla1}.
\item If $U=\singleton{u}$ is a singleton and $U'$ is infinite, then $(x,x')$ has a colour in $[r]$ because otherwise, $x' \in N^+_{r+1}(u)\cap U'\subseteq W$ by \ref{bla2}.  
\item If $U$ is infinite and $U'=\singleton{u'}$ is a singleton, then $(x,x')$ has a colour in $[r]$ because otherwise, $x \in N^-_{r+1}(u')\cap U\subseteq W$ by \ref{bla3}. \qedhere
\end{itemize}

\end{proof}

Let $T$ be a tournament as in Claim $3$. Since $c$ induces an $r$-edge colouring of $T$ such that there are no directed paths of length $\ell_i$ for any $i \in [r]$, it follows from Theorem \ref{po-shenloh thm} that
$|\mathcal{S}'|+|\U\backslash\mathcal{S}|=|T|\le\prod_{i\leq r}\ell_i$. But now we have $k=|\mathcal{S}'|$ pairwise disjoint monochromatic directed paths in colour $r+1$ in $\K{\N}[S]$ covering $S$ and---by Lemma \ref{lemma strongly connected then path}---we can cover each infinite partition class $U \in \U \setminus \mathcal{S}$ with a single one-way infinite monochromatic directed paths in colour $r+1$. Thus, we have covered the vertex set of $\K{\N}$ by $k + |\U \setminus \mathcal{S}| \leq \prod_{i\leq r}\ell_i$ many monochromatic directed paths in colour $r+1$, completing the proof of the theorem.
\end{proof}

\bibliographystyle{plain}
\bibliography{RedAndBluePaths}
\end{document}